\documentclass[12pt]{amsart}

\usepackage{amsmath,amsfonts,amssymb}
\usepackage{amsthm,amsfonts,amssymb}
\usepackage{graphics}
\usepackage{pifont}
\usepackage[usenames,dvips]{color}
\usepackage{amsrefs}
\def\({\left(}
\def\){\right)}
\def\be {\begin{equation}}
\def\en{\end{equation}}

\let\nc\newcommand

\nc{\Ref}[1]{{\rm(\ref{#1})}}

\nc{\BB}{{A_{p,q}}} \nc{\bean}{\begin{eqnarray}} \nc{\eean}{\end{eqnarray}}
\nc{\bea}{\begin{eqnarray*}} \nc{\eea}{\end{eqnarray*}}

\def\Bz{\mathbb Z}
\def\Bn{\mathbb N}

\theoremstyle{plain}
\begingroup 
\newtheorem{thm}{Theorem}[section]   
\newtheorem*{thm*}{Theorem}          
\newtheorem*{cor*}{Corollary}        

\newtheorem{lem}[thm]{Lemma}         
\endgroup

\theoremstyle{definition}

\newtheorem*{rem*}{Remark}
\newtheorem*{ack*}{Acknowledgment}

\theoremstyle{remark}
\newtheorem{rem}[thm]{Remark}        %

\theoremstyle{definition}

\numberwithin{equation}{section}

\begin{document}

\title[Delannoy adic]
{an adic dynamical system related to the Delannoy numbers}

\dedicatory{In memory of Dan Rudolph---friend, colleague, and inspiration.}

\author{Karl Petersen}
\address{Department of Mathematics, University of North Carolina
at Chapel Hill\\ Chapel Hill, NC 27599-3250, USA} \email{petersen@math.unc.edu}

\date{\today}


\begin{abstract}
We introduce an adic (Bratteli-Vershik) dynamical system based on a diagram whose path counts from the root are the Delannoy numbers. We identify the ergodic invariant measures, prove total ergodicity for each of them, and initiate the study of the dimension group and other dynamical properties.
\end{abstract}

\maketitle

\section{Introduction}\label{sec:intro}

Adic, or Bratteli-Vershik, dynamical systems present the cutting and stacking constructions of ergodic theory in a form that is convenient for studying important relations such as orbit equivalence; they also draw connections with the theory of $C^*$ algebras and group representations and raise interesting questions in combinatorics and number theory. In recent years some progress has been made on the dynamics of several natural non-stationary and non-simple adic systems, such as the Pascal, Euler, reverse Euler, and Stirling systems. See \cite{FP2} for a discussion of how such systems arise from walks and reinforced walks on graphs and \cites{PetersenSchmidt1997, AdamsPetersen1998, MelaPetersen2005, Mela2006, Frick2009, delaRueJanvresse2004,BKPS2006, FP, PetersenVarchenko2010} and the references that these papers contain for background on these and other adic systems. In this paper, by investigating the dynamical properties of a particular example, which we call the {\em Delannoy adic}, we initiate study of a class of adic systems which, though non-stationary, still are highly regular in that each vertex has the same pattern of exiting edges. We begin with a brief discussion of how we came across this particular diagram.

In a music theory seminar at IRCAM in Paris, Norman Carey and David Clampitt discussed a diagram attributed to Nicomachus, and later studied by Nicole Oresme, which was supposed to have wonderful properties, including forming a basis for the generation of various kinds of musical scales \cite{CareyClampitt1996}. The diagram, with entries $2^n 3^k$ at the points $(n,k)$ of the integer lattice in the plane, is shown in Figure \ref{fig:Nichomachus}.

\begin{figure}[!h]
  \scalebox{.60}{\includegraphics{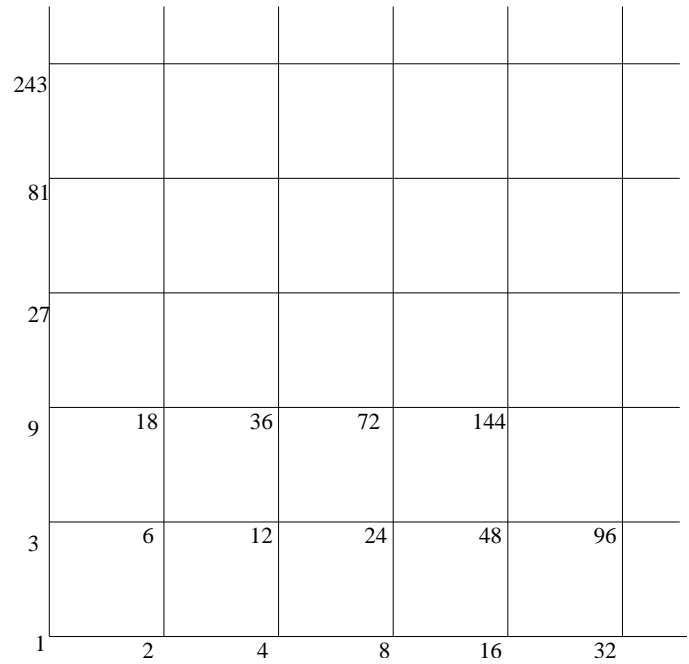}}
  \caption{The first part of the Nicomachus diagram.
  }
  \label{fig:Nichomachus}
  \end{figure}

  Naturally one may ask whether this might turn into a Bratteli diagram in such a way that the entries $2^a 3^b$ would be the dimensions of vertices, that is, the numbers of paths from the root (lower left corner) to the vertices. Adding diagonals to the boxes makes this happen, as seen in Figure \ref{fig:NichomachusBrat}.

  \begin{figure}[!h]
  \scalebox{.60}{\includegraphics{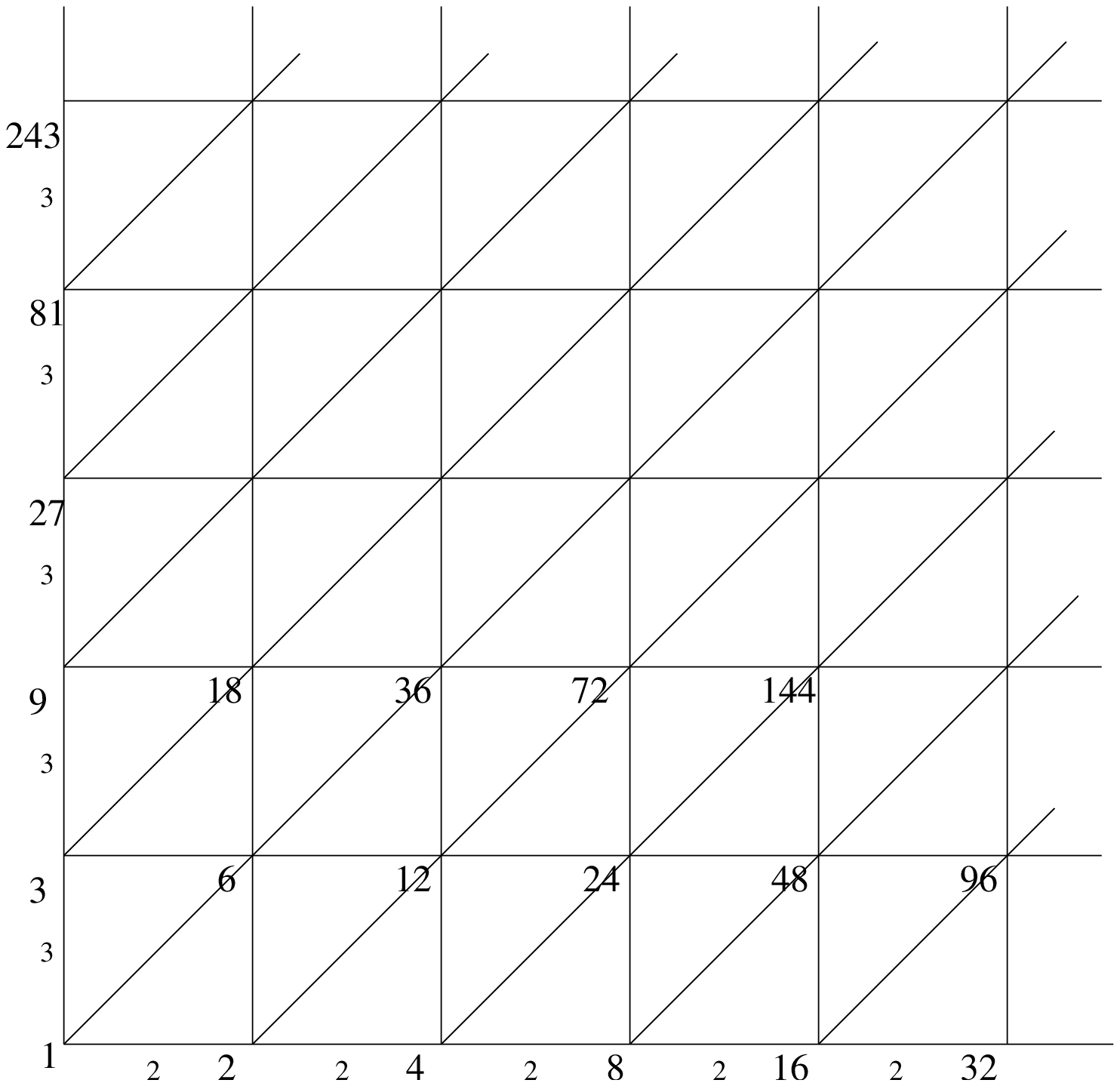}}
  \caption{The Nicomachus diagram with added diagonals.
  Numbers next to edges indicate multiplicities greater than 1.}
  \label{fig:NichomachusBrat}
  \end{figure}

   As we will see below, this does not produce an especially interesting adic dynamical system, since the only ergodic invariant measures are supported on the two boundary odometers. Reducing the edge multiplicities to 1 produces a figure whose path counts are the {\em Delannoy numbers} \cite{BanderierSchwer2005} $D(n,k)$, for $n,k\geq 0$; see Figure \ref{fig:Del1}.

\begin{figure}[!h]
  \scalebox{.60}{\includegraphics{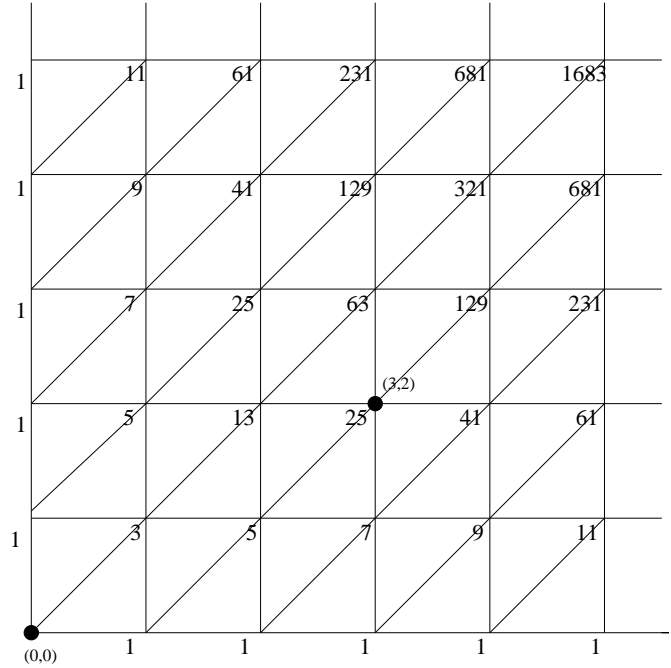}}
  \caption{The Delannoy graph.}
  \label{fig:Del1}
  \end{figure}

  Notice, however, that this is not a legitimate Bratteli diagram, since it is not true that there are connections only from one level (the set of vertices a fixed graph distance from $(0,0)$) to the next; see Figure \ref{fig:delDiag}.

\begin{figure}[!h]
  \scalebox{.60}{\includegraphics{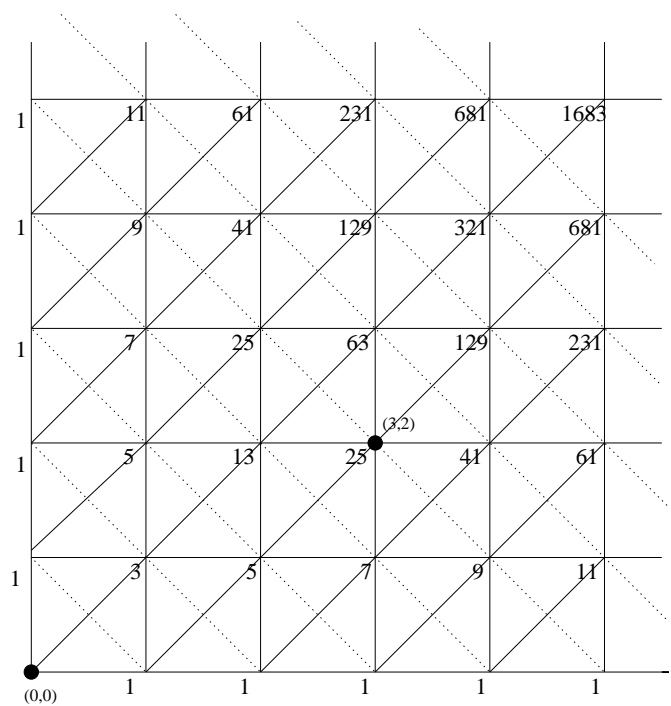}}
  \caption{Levels in the Delannoy graph.
  }
  \label{fig:delDiag}
  \end{figure}

  This can be remedied by inserting ``extra'' vertices at the centers of the squares. While this complicates the graph, the space of infinite paths beginning at the root is the same as before, and so we may regard this as a Bratteli diagram and adic dynamical system that realizes the Delannoy numbers as the dimensions of vertices; see Figure \ref{fig:DelBrat}.

  \begin{figure}[!h]
  \scalebox{.60}{\includegraphics{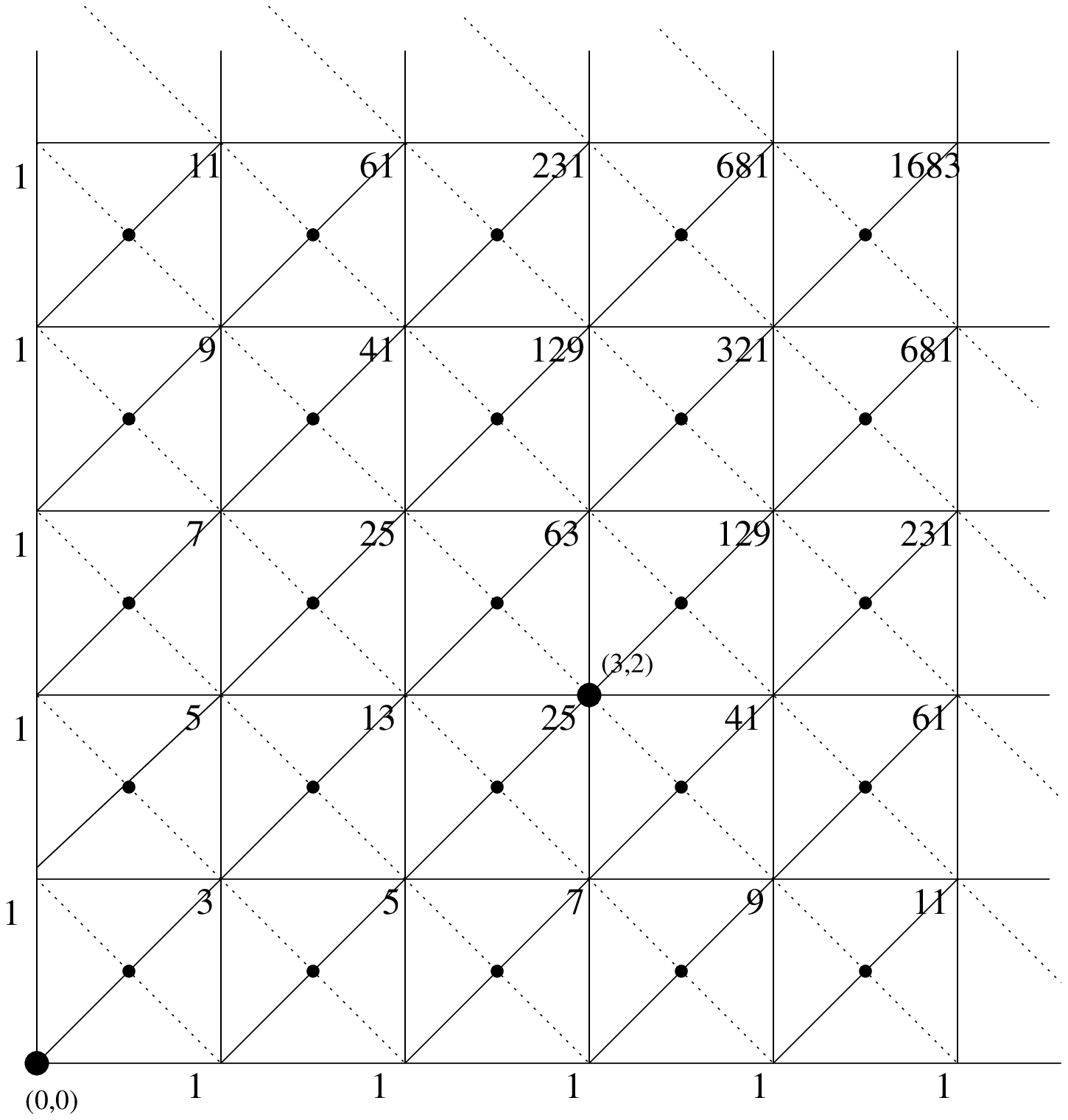}}
  \caption{The Delannoy graph made into a Bratteli diagram.
  }
  \label{fig:DelBrat}
  \end{figure}

  Here are some known key formulas involving the Delannoy numbers (see, for example, \cites{Sloane2010, SloanePlouffe1995}).
  \be
   \begin{gathered}
   D(n,0)=D(0,n)=1 \text{ for all } n\geq 0;\\
   D(n,k)=0 \text{ if either } n \text{ or } k < 0;\\
   D(n,k)=D(n,k-1)+D(n-1,k-1)+D(n-1,k) \text{ for all } n,k.
   \end{gathered}
   \en

   \be
   \sum_{n,k \geq 0} D(n,k) x^n y^k = \frac{1}{1-(x+y+xy)} .
   \en

  Assuming $n \geq k$,
   \be \label{eq:DelFormulas}
   \begin{gathered}
   D(n,k)= \sum_{d=0}^{k} \binom{k}{d}  \binom{n+k-d}{k} = \sum_{d=0}^{k} 2^d \binom{n}{d} \binom{k}{d}\\
  = \sum_{d=0}^k \binom{k}{d} \binom{n+d}{k} = \sum_{d=0}^k \binom{k}{k-d} \binom{n+d}{k}
  \\
  = \sum_{d=0}^k \binom{n+k-d}{k-d} \binom{n}{d} = \sum_{d=0}^k \binom{n+d}{d} \binom{n} {k-d} .
   \end{gathered}
   \en

   \be
   D(n,n) \sim (3 + 2 \sqrt{2})^n(.57 \sqrt{n} - .067 n^{-3/2} + .006 n^{-5/2} + \dots ) \text{ as } n\to \infty .
   \en

 Now we define the two dynamical systems which concern us. Each is based on a countably infinite directed graph with an ordering on the edges that enter each vertex. The {\em Nicomachus graph} (shown in Figure \ref{fig:NichomachusBrat}) has vertices $(n,k) \in \Bz_+^2$ and an edge from $(n,k)$ to $(n+1,k)$, to $(n,k+1)$, and to $(n+1,k+1)$ for all $n,k \geq 1$. From each vertex $(n,0), n \geq 0$, there are two edges to $(n+1,0)$ one edge to $(n+1,1)$, and, if $n \geq 1$, one edge to $(n,1)$. From each vertex $(0,k), k \geq 0$, there are three edges to $(0,k+1)$, one edge to
 $(1,k+1)$, and, if $k \geq 1$, one edge to $(1,k)$. Each doublet or triplet of edges entering a vertex $(n,0)$ or $(0,k)$ is ordered arbitrarily. The three edges entering a vertex $(n+1,k+1)$ with $n,k \geq 0$ are ordered so that the one from $(n+1,k)$ is the first and the one from $(n,k+1)$ is the last.

 The {\em Delannoy graph} (shown in Figure \ref{fig:Del1}) has vertices $(n,k) \in \Bz_+^2$ and an edge from $(n,k)$ to $(n+1,k)$, to $(n,k+1)$, and to $(n+1,k+1)$ for all $n,k \geq 0$. Thus it is obtained from the Nicomachus graph by replacing the multiple edges along the $x$ and $y$ axes by single edges. Again the three edges entering a vertex $(n+1,k+1)$ with $n,k \geq 0$ are ordered so that the one from $(n+1,k)$ is the first and the one from $(n,k+1)$ is the last.

 In each case the phase space $X$ of our dynamical system will consist of infinite paths $x$ that begin at the origin. The paths have labelings, either corresponding uniquely to their sequences of edges $e_i=e_i(x), i \geq 1$ (with labels from an alphabet $A=\{h,d,v\}$, for horizontal, diagonal, or vertical, respectively, and $A=\{h,d,v,h_1,h_2,v_1,v_2,v_3\}$ when there can be more than one edge between a pair of vertices), or to the sequences $(n_i,k_i)=(n_i(x),k_i(x)), i \geq 0,$ of vertices through which they pass. The space
 $X$ is a compact totally disconnected metric space when we define the distance from $x$ to $y$ to be less than or equal to $1/2^N$ if $e_i(x) = e_i(y)$ for all $i \leq N$. The topology of $X$ is generated by the family of (clopen) {\em cylinder sets} defined by finite sequences $(a_i)$ from $A$,
 \be
 [a_1a_2 \dots a_I] = \{x \in X: e_i(x) = a_i \text{ for all } i \leq I\}.
 \en

 Two paths $x,y \in X$ are {\em tail equivalent} if they coincide from some point on: there is a smallest $N$ such that $e_i(x)=e_i(y)$ for all $i > N$. In such a case we define $x < y$ if $e_N(x) < e_N(y)$. Define $X_{\max}$ to be the set of maximal paths, i.e. those that include only maximal edges, and define $X_{\min}$ to be the set of minimal paths. For each of the two graphs, $X_{\max} \cup X_{\min}$ is countable. The {\em adic transformation} $T: X \setminus X_{\max} \to X \setminus X_{\min}$ is defined by letting $Tx$ equal the smallest $y \in X$ such that $y > x$. Then $T$ is a homeomorphism. We define $T$ and $T^{-1}$ to be the identity on each of $X_{\max}$ and $X_{\min}$. As extended, $T$ is not continuous on $X$. The orbits of $T$ coincide with the equivalence classes of the tail relation.

  \begin{ack*}
The author thanks Jean-Paul Allouche, James Damon, Sarah Bailey Frick, and Fran\c{c}ois Ledrappier for helpful discussions; the University of Paris 6, the Wroc{\l}aw University of Technology, and the Institut Henri Poincar\'{e} for hospitality while this research was underway; and the referees for suggestions that improved the exposition.
\end{ack*}

  \section{Invariant measures}\label{sec:invmeasures}

In the following, by ``measure" we will mean ``Borel probability measure".
For an invariant measure of an adic system, all cylinder sets determined by initial path segments from the root $R$ to a fixed vertex $v$ have the same measure. For vertices $u$ and $v$ in the diagram of the system, denote by $\dim (u,v)$ the number of paths in the diagram from $u$ to $v$. From the Ergodic Theorem it follows that if $\mu$ is an ergodic invariant measure and $C$ is a cylinder set determined by an initial path segment ending at a vertex $v$, then
 \be
 \mu (C) =\lim_{n \to \infty}  \frac{\dim (v,x_n)}{\dim (R,x_n)}
 \en
 for $\mu$-a.e. path $x=(x_n)$ (indexed here by its vertices).
In the case of the Delannoy graph
it is enough to specify the measure of each cylinder set determined by an initial path segment that ends at a ``regular'' vertex $(n_0,k_0)$ in the integer lattice, since the measure of any cylinder ending at one of the ``extra'' vertices $(n_0+1/2,k_0+1/2)$ equals that of any cylinder ending at $(n_0+1,k_0+1)$. We will abbreviate by saying such a cylinder set ``ends at $(n_0,k_0)$''.

Thus to identify the ergodic measures, it is important to have information about path counts between vertices in the diagram that defines the system.
 The asymptotics of Delannoy numbers were determined by Pemantle and Wilson \cite{PemantleWilson2002} by analytic methods:

\be \label{eq:PemWil}
D(n,k) \sim \left( \frac{\sqrt{n^2 + k^2}-k}{n}\right)^{-n} \left( \frac{\sqrt{n^2+k^2}-n}{k}\right)^{-k}  \sqrt{ \frac{nk/(2 \pi)}{(n+k-\sqrt{n^2+k^2})^2 \sqrt{n^2+k^2}}} ,
\en
uniformly if $n/k$ and $k/n$ are bounded.

\begin{thm}\label{thm:NicInvMeas}
 The only ergodic (invariant probability) measures for the Nicomachus adic dynamical system (described by Figure \ref{fig:NichomachusBrat}) are the two unique measures supported on the two boundary odometers.
\end{thm}
\begin{proof}
Let $(n_0,k_0)$ be an interior regular vertex, i.e. $n_0,k_0>0$. We claim that
\be
 \frac{\dim((n_0,k_0),(n_0+n,k_0+k))}{\dim((0,0),(n_0+n,k_0+k))}= \frac{D(n,k)}{2^{n_0+n} 3^{k_0+k}} \to 0 \quad\text{ as } n+k \to \infty .
 \en
 It follows that any ergodic invariant measure must assign measure 0 to any cylinder set determined by an initial path segment that ends at an interior vertex.

    Let $ 0<\epsilon <1$. We consider first $(n,k)$ with $\theta = k/n \in [\epsilon, 1]$. According to the Pemantle-Wilson estimate, for large $n,k$, uniformly in this region,
    \be
    D(n,k) \sim (\sqrt{1+\theta^2}+\theta)^n \left(\frac{\sqrt{1+\theta^2}+1}{\theta}\right)^{\theta n} \frac{\sqrt{\theta /(2 \pi n)}}{\sqrt{(1+\theta-\sqrt{1+\theta^2})^2\sqrt{1+\theta^2}}} .
    \en
    Let
    \be
    \begin{gathered}
    A(\theta)=(\sqrt{1+\theta^2}+\theta)\left(\frac{\sqrt{1+\theta^2}+1}{\theta}\right)^\theta ,
    \quad\ G(\theta)= \log A(\theta) -\log 2 - \theta \log 3 , \\
    \text{ and } \quad B(\theta) = \frac{\sqrt{\theta }}{\sqrt{(1+\theta-\sqrt{1+\theta^2})^2\sqrt{1+\theta^2}}} .
    \end{gathered}
    \en
    We note that $B(\theta)$ is bounded for $\theta \in [\epsilon, 1]$ and $G(\theta)$ has an absolute maximum at $\theta=3/4$, so that
    \be
    \frac{A(\theta)}{2 \cdot 3^\theta} \leq \frac{A(3/4)}{2 \cdot 3^{3/4}} = 1
    \en
    for all $\theta \in [\epsilon,1]$. Hence
    \be
        \frac{D(n,k)}{2^n 3^k} \leq c \, \left( \frac{A(\theta)}{2 \cdot 3^\theta} \right)^n \frac{1}{\sqrt{n}}\to 0
        \en
        uniformly as $n \to \infty$ in this region.

    Second, we consider $(n,k)$ in the region where $k \leq \epsilon n$.
    Let
    \be
    H(\epsilon)= -\epsilon \log\epsilon - (1-\epsilon) \log (1-\epsilon)   \quad\text{ and } \lambda = e^{H(\epsilon)}.
    \en
    Note that $0 \leq H(\epsilon) \leq \log 2$ for all $\epsilon \in [0,1]$ and $\lambda \to 1$ as $\epsilon \to 0^+$.
    By Stirling's Formula,
     \be
     \begin{gathered}
     \binom{n}{\epsilon n} \sim \lambda^n \quad\text{ and } \\
     \binom{ \epsilon n}{d} \leq \binom{\epsilon n}{\epsilon n/2}
      \sim \frac{2^{\epsilon n}}{\sqrt{\pi \epsilon n/2}} \quad\text{ for } d \leq k \leq \epsilon n .
      \end{gathered}
     \en
     Thus, since in this region $D(n,k)$ and $C(n,k)$ are increasing in $k$,
         \be
         \begin{aligned}
    D(n,k) &=  \sum_{d=0}^{k} 2^d \binom{n}{d} \binom{k}{d}
    \leq D(n, \epsilon n) \leq \sum_{d=0}^{\epsilon n} 2^d \lambda^n \binom{\epsilon n}{d} \\
    &\leq \epsilon n 2^{\epsilon n} \lambda^n \frac{2^{\epsilon n}}{\sqrt{\pi \epsilon n/2}}
    \sim c \sqrt{\epsilon n} \lambda^n 2^{2 \epsilon n} = c \sqrt{\epsilon n}(2^{2 \epsilon} \lambda)^n.
    \end{aligned}
    \en
    Choosing $\epsilon$ small enough that
    \be
    2^{2 \epsilon} \lambda < 2 ,
    \en
    we have that $D(n,k)/2^n \to 0$ uniformly in this region as $n \to \infty$.

    In the region $k \geq n$, similar but easier estimates apply, since $D(n,k)$ is symmetric in $n$ and $k$, and $2^n 3^k$ is now much larger than before.
\end{proof}

\begin{thm}\label{thm:DelInvMeas}
 The non-atomic ergodic (invariant probability) measures for the Delannoy adic dynamical system (described by Figure \ref{fig:Del1} or Figure \ref{fig:DelBrat}) are a one-parameter family $\{ \mu_\alpha: \alpha \in [0,1]\}$ given by choosing nonnegative $\alpha, \beta, \gamma$ with $\alpha+\beta+\gamma=1$ and $\beta \gamma = \alpha$ and then putting weight $\beta$ on each horizontal edge, weight $\gamma$ on each vertical edge, and weight $\alpha$ on each diagonal edge. (The measure of any cylinder set is then determined by multiplying the weights on the edges that define it.)
\end{thm}
\begin{proof}
  Each $\mu_\alpha$ as above is adic-invariant (by the conditions). It describes an i.i.d. random walker who starts at $(0,0)$ and at each time $t=0,1,2,\dots$ can choose independently whether to take a diagonal, horizontal, or vertical step. The probability of each finite initial path segment is the probability of the cylinder set that it defines. We will give two proofs that each $\mu_\alpha$ is ergodic, first by means of the Hewitt-Savage 0,1 Law (see \cite{Billingsley1995}), and then by a collision argument as in \cite{BKPS2006}, using known conditions for recurrence of random walks.

(1) Identify the set $X$ of infinite paths from the root in the Delannoy graph with $\{h,d,v\}^{\mathbb N}$ by using the edge labels. Then $\mu_{\alpha}$ is a Bernoulli measure for the stationary process in which the symbols $h,d,v$ arrive independently with probabilities $\beta, \alpha, \gamma$, respectively. If $A \subset X$ is a Borel set that is invariant under the adic transformation $T$, then it is also invariant under permutations of finitely many coordinates. Therefore the $\sigma$-algebra of adic-invariant Borel sets is contained in the $\sigma$-algebra of symmetric sets (those which are invariant under permutations of finitely many coordinates). Since the latter algebra is trivial by the Hewitt-Savage Theorem, $A$ has measure $0$ or $1$, and thus
  the measure $\mu_\alpha$ is ergodic.

 (2) The ``collision argument'' presented in \cite{BKPS2006} establishes ergodicity of $\mu_\alpha$ once we show that for $\mu_\alpha \times \mu_\alpha$-a.e. pair of infinite paths $(x,y) \in X \times X$, with ``regular'' vertices $(n_i,k_i)$ and $(r_i,s_i)$ respectively, there are infinitely many $i$ for which $(n_i,k_i)=(r_i,s_i)$.

 As mentioned above, paths $x \in X$ are trajectories of a random walk in $\Bz^2$ with independent increments $(1,1), (0,1), (1,0)$ having probabilities $\alpha, \beta, \gamma$, respectively. We want to show recurrence (i.e., infinitely many returns to the origin of the symmetric (mean $0$) process $Z=X-X$. This process has independent increments as shown in the first array below, each with probability given in the second array below:
 \be
 \left(
 \begin{matrix}
 (0,0) & (0,1) & (1,0) \\
 (0,-1) & (0,0) & (1,-1) \\
 (-1,0) & (-1,1) & (0,0)
 \end{matrix}
 \right)
 \quad\qquad
 \left(
 \begin{matrix}
 \alpha^2 & \alpha\beta & \alpha\gamma \\
 \beta\alpha & \beta^2 & \beta\gamma \\
 \gamma\alpha & \gamma\beta & \gamma^2
 \end{matrix}
 \right)
 \en
 (the rows and columns are each indexed by the possible increments listed above). Since the walk in $\Bz^2$ has mean 0 and finite second moment, it is recurrent  \cite{Spitzer1964}*{p. 83}.

(3)  Now we prove that every adic-invariant ergodic measure for the Delannoy system is Bernoulli for the shift, i.e. is one of the measures $\mu_\alpha$ described above,
by applying the isotropy arguments of \cite{Mela2006}. Let $\mu$ be an adic-invariant ergodic measure. Since $\mu$ is non-atomic and adic-invariant,  $\mu[t]>0$ for each $t \in A=\{d,h,v\}$. Let $C=[c_0 \dots c_n]$ be a cylinder set determined by a string $c_0 \dots c_n$ on the alphabet $A$. For a finite or infinite string $w$ on $A$ and $t \in A$, let $\sigma_t w=tw$. Denote $C^t=[c_0 \dots c_nt]$, and note that $\sigma_t C$ and $C^t$ end at the same vertex. For a vertex $y=(n,k)$, let $y^h=(n-1,k), y^v=(n,k-1)$, and $y^d=(n-1,k-1)$.

First we show that the given ergodic measure $\mu$ is shift invariant. Let $x \in X$ be an infinite path labeled by a string $(\omega_i) \in A^{\Bz_+}$. Denote by $x_i=(n_i,k_i)$ the terminal vertex of the path $\omega_0 \dots \omega_i$. The set of paths from the terminal vertex of $C$ to $x_i$ breaks into three disjoint sets, depending on whether the last entry of the string labeling the path is $d, h$, or $v$. Thus
\be
\dim(C,x_i)=\dim(C,x_i^d) + \dim(C,x_i^h) + \dim(C,x_i^v) .
\en
By isotropy of the graph, for each $t \in A$,
\be
\dim(C,x_i^t) = \dim(\sigma_t C, x_i) ;
\en
and
\be
\sigma^{-1}C = \sigma_d C \cup \sigma_h C \cup \sigma_v C \quad\text { (disjoint union) }.
\en
Thus dividing the equation
\be
\dim(C,x_i)= \dim(\sigma_d C,x_i)+\dim(\sigma_h C,x_i) + \dim(\sigma_v C,x_i)
\en
through by $\dim(x_i)$ and letting $i \to \infty$, if $x$ is generic for $\mu$, we conclude that $\mu (C) = \mu(\sigma^{-1}C)$.

Now we will show that $\mu(C^t)/\mu(C)$ is independent of the cylinder set $C$ for each $t \in A$, thereby concluding the proof.
First, we show that
 \be\label{Eq:sameratio}
 \frac{\mu(C^t)}{\mu(C)} = \frac{\mu(C^{st})}{\mu(C^s)} \quad\text{ for all } s,t \in A .
 \en

 Note that by isotropy (it doesn't matter where the extra step is inserted, we still arrive at the same terminal vertex),
 \be
 \dim(C^t,x_i) = \dim(C^{st},(\sigma_sx)_{i+1}), \qquad \dim(C,x_i) = \dim(C^s, (\sigma_sx)_{i+1}) .
 \en
 Dividing, we have
 \be
 \frac{\dim(C^t,x_i)}{\dim(C,x_i)} = \frac{\dim(C^{st},(\sigma_sx)_{i+1})}{\dim(C^s, (\sigma_sx)_{i+1})} .
 \en
 Divide top and bottom on the left by $\dim(x_i)$, on the right by $\dim((\sigma_sx)_{i+1})$, and let $i \to \infty$. As in \cite{Mela2006}, if $E$ is the set of $\mu$-measure 1 of points $x$ that give the correct values for the limits on the left-hand side, then $\mu[s]>0$ implies that $\mu(\sigma_s E)>0$, and hence $\mu(\sigma_s E \cap E) >0$. Thus using any $x \in \sigma_s E \cap E$ yields Formula (\ref{Eq:sameratio}).

 We can then compute that
 \be
 \frac{\mu(C^t)}{\mu(C)} = \frac{\mu[c_0\dots c_{n-1}]^{c_nt}}{\mu[c_0\dots c_{n-1}]^{c_n}} = \frac{\mu[c_0\dots c_{n-1}]^{t}}{\mu[c_0\dots c_{n-1}]}= \dots = \frac{\mu[c_0]^t}{\mu[c_0]} .
 \en
 By adic invariance of $\mu$, for $r,s,t \in A$,
 \be
 \mu[st]=\mu[ts] \qquad\text{ and } \mu[rst]=\mu[srt] ,
 \en
 since the paths defining the cylinders end at the same vertex in each case.
 Then
 \be
 \frac{\mu[s]^t}{\mu[s]} = \frac{\mu[s]^{rt}}{\mu[s]^r} = \frac{\mu[srt]}{\mu[sr]}=\frac{\mu[rst]}{\mu[rs]}=\frac{\mu[r]^{st}}{\mu[r]^s}
 =\frac{\mu[r]^t}{\mu[r]},
 \en
 so that $\mu[c_0]^t/\mu[c_0]$ is independent of $c_0$.
 \end{proof}

\begin{rem}
Suppose that $\mu_{\alpha}$ is a measure on the Delannoy system as described in the statement of Theorem \ref{thm:DelInvMeas}.
Suppose that $C$ is a cylinder set determined by a path segment from $(0,0)$ to a regular vertex $(n_0,k_0)$. Because $\mu_\alpha$ is Bernoulli, hence invariant and ergodic for the shift on the space of infinite one-sided sequences on the alphabet of three symbols $\{d,h,v\}$ (for diagonal, horizontal, or vertical steps taken by the random walker), for $\mu_\alpha$-a.e. infinite path $x=(n_i,k_i)$ starting at $(0,0)$ in the Delannoy graph, we have
 \be
 k_i/n_i \to \rho = \frac{\alpha+\gamma}{\alpha+\beta}.
 \en
 From the ergodicity of $\mu_{\alpha}$, which we have proved above, it follows that for a.e. path $(n_i,k_i)$,
 \be
 \frac{\dim((n_0,k_0),(n_i,k_i))}{\dim((0,0),(n_i,k_i))} \to (\sqrt{1+\rho^2}+\rho)^{-n_0} \left( \frac{\sqrt{1+\rho^2}+1}{\rho} \right)^{-k_0} \quad\text{ as } i \to \infty .
 \en
 It is not clear that the Pemantle-Wilson asymptotics (\ref{eq:PemWil}) yield this conclusion for all paths with $k_i/n_i \to \rho$. Similarly, it is not clear that the asymptotics would allow one to calculate the measure of any cylinder extended by a single edge and conclude that every ergodic measure is one of the measures $\mu_{\alpha}$.
 \end{rem}

\section{Total ergodicity}\label{sec:TotErg}

  In this section we will show that with respect to each of its ergodic measures the Delannoy adic system is {\em totally ergodic}, i.e. has no proper roots of unity among its $L^2$ eigenvalues. The proof depends on congruence properties of the Delannoy numbers with respect to primes. If $\lambda$ is an eigenvalue of $T$, then approximating an eigenfunction by linear combinations of characteristic functions of cylinders, many of whose points return after $\dim(R,x_n)$ iterates, shows that
  \be \label{eq:evalue}
  \lambda^{D(n_i,k_i)}\to 1 \qquad \text{ for a.e. path } x=((n_i,k_i)) .
  \en
  If $\lambda = e^{2 \pi i m/n}$ with $m,n$ relatively prime, if $p$ is a prime that divides $n$, and if a.e. path $x=((n_i,k_i))$ includes infinitely many vertices for which $D(n_i,k_i)$ is not divisible by $p$, then it is impossible for the convergence in (\ref{eq:evalue}) to hold---see \cite{Petersen2002}. We will find ``blocking sets": configurations of vertices with $D(n,k)$ not divisible by $p$ that must be hit infinitely many times by each infinite path in the diagram.

  A basic ingredient in the following argument is the well-known formula of Lucas \cite{Lucas1878}: If $p$ is a prime and $n=n_0 + n_1 p + n_2 p^2 + \dots, k=k_0 + k_1 p + k_2 p^2 + \dots ,$ are the base $p$ expansions of $n$ and $k$ (each $n_i, k_i \in \{0,1,\dots, p-1\}$) then
  \be
  \binom{n}{k} \equiv_p \binom{n_0}{k_0} \binom {n_1}{k_1} \dots .
  \en

\begin{lem}\label{lem:PrimeCong}
For $p$ prime, $r \geq 0$, and $j=0,1,\dots ,p^r-1$,
\be
C(p^r-1,j)=\binom{p^r-1}{j} \equiv_p (-1)^j .
\en
\end{lem}

\begin{proof}
$C(p^r-1,j)\equiv_p1$ at $j=0$ and $j=p^r-1$, and $C(p^r-1,j)+C(p^r-1,j+1) \equiv_p 0$ for $j=0, 1, \dots ,p^r-2$ since, by Lucas' formula, $C(p^r,j) \equiv_p 0$ for $1 < j < p^r$.

 Alternatively,
 \be
 (i+1)C(p-1,i+1) = (p-(i+1)) C(p-1,i) \equiv_p -(i+1) C(p-1,i) \text{ for } i<p-1 ,
 \en
 so
 \be
 C(p-1,i+1) \equiv_p -C(p-1,i) \text{ if } i<p-1 .
 \en
 Thus, again by Lucas' formula, increasing any entry $i$ of the expansion of $j$ base $p$ multiplies $C(p^r-1,j)$ mod $p$ by $-1$. And if the entries of the expansion of $j$ base $p$ are $(p-1, p-1, \dots , 0, j_s)$ with $j_s < p-1$, then the entries of $j+1$ base $p$ are $(0, 0, \dots, 0, j_s+1)$, and $C(p-1,p-1)=C(p-1,0)$, so again the sign changes when we pass from $j$ to $j+1$.
 \end{proof}

\begin{lem}\label{lem:BinCong}
For $p$ prime, $r \geq 0$, $j=0,1,\dots ,p^r-1$, and $i=0,1,\dots$,
\be
\binom{j+ip^r}{p^r-1} \equiv_p \binom{j}{p^r-1} .
\en
\end{lem}
\begin{proof}
Let $k=p^r-1=(p-1) + (p-1)p= \dots + (p-1)p^{r-1}$ and $j=j_0+j_1p + \dots +j_{r-1}p^{r-1}+j_rp^r+\dots$, with all $j_s\in \{0,1,\dots,p-1\}$. (So the base $p$ entries of $k$ are $k_m=p-1$ if $m = 0, 1, \dots, r-1$, $k_m=0$ for $m>r-1$.)

Then for $i=0,1,\dots$, we have
\be
j+ip^r = j_0 + j_1p + \dots +j_{r-1}p^{r-1} + \dots .
\en
When we pass from $j$ to $j+ip^r$, in the base $p$ expansion only entries $j_m$ for $m \geq r$ can change, say to $j_m'$. But for $m \geq r$, $C(j_m',0)=C(j_m,0)=1$. Thus $C(j_m,k_m)=C(j_m',k_m)$ for all $m$, and hence, by Lucas' formula, $C(j+ip^r,k)\equiv_p C(j,k)$.
\end{proof}

\begin{lem}\label{lem:MoreCong}
For $p$ prime, $r \geq 0$, and $j=1,\dots ,p^r-1$,
\be
\binom{p^r-1+j}{p^r-1} \equiv_p 0.
\en
\end{lem}
\begin{proof}
Let $p$ be prime, $k=p^r-1$, $j \in \{0,1,\dots,p^r-1\}$. Let $k_m, j_m$ be the entries in the base $p$ expansions of $k$ and $j$, respectively, as in the proof of Lemma \ref{lem:BinCong}, so that some (first) $j_m>0$.

Then $p^r+j=j_0 + j_1 p + \dots + j_{r-1} p^{r-1} + p^r$, and when we form $p^r+j-1$, we decrease the first nonzero $j_m$ (which is in $\{1,\dots,p-1\}$), so that the base $p$ expansion of $p^r+j-1$ has an entry $j_i'=j_i-1<p-1$, whence $C(j_i',p-1)=0$.
\end{proof}

\begin{thm}\label{thm:DelCong}
For $p$ prime, $r \geq 0$, and $n=0,1,2, \dots$,
\be
D(n,p^r-1) \equiv_p (-1)^{(n \mod p^r)} .
\en
\end{thm}
\begin{proof}
 Let $k=p^r-1$ and $m= (n \mod p^r) \in \{0,1,\dots,p^r-1\}$. Using Lemma \ref{lem:BinCong},
 \be
 D(n,k) = \sum_{d=0}^k C(k,d) C(n+d,k) \equiv_p \sum_{d=0}^k C(k,d) C(m+d,k)= \sum_{d=0}^k C(k,d) C(m+k-d,k).
 \en
 By Lemma \ref{lem:MoreCong}, the only nonzero term mod $p$ is at $m-d=0$. Thus, if $n=\sum m_i p^i$ with each $m_i \in \{0,1, \dots, p-i\}$, the sum reduces to
 \be
 C(k,m) \equiv_p C(p-1,m_0) \cdots C(p-1, m_{r-1}) \neq 0 .
 \en
 In fact, by Lemma \ref{lem:PrimeCong}, $C(k,m) \equiv_p \pm 1$.
 \end{proof}

 \begin{thm}\label{thm:TotalErg}
 With respect to each of its ergodic (invariant probability) measures, the Delannoy adic dynamical system is totally ergodic (i.e., has among its eigenvalues no roots of unity besides $1$).
 \end{thm}
 \begin{proof}
 As in the proof of total ergodicity of the Pascal adic (see, for example, \cites{PetersenSchmidt1997, AdamsPetersen1998, Petersen2002}), it is enough to show that for each prime $p$, every infinite path in the diagram must pass through infinitely many vertices whose dimensions are not divisible by $p$. The vertices $\{ (n,p^r-1): n,r \geq 0\}\cup\{(p^r-1,k): k,r \geq 0\}$, form such a ``blocking set"---see Figure \ref{fig:DelBlocking}.
 \end{proof}

 \begin{figure}[!h]
  \scalebox{.60}{\includegraphics{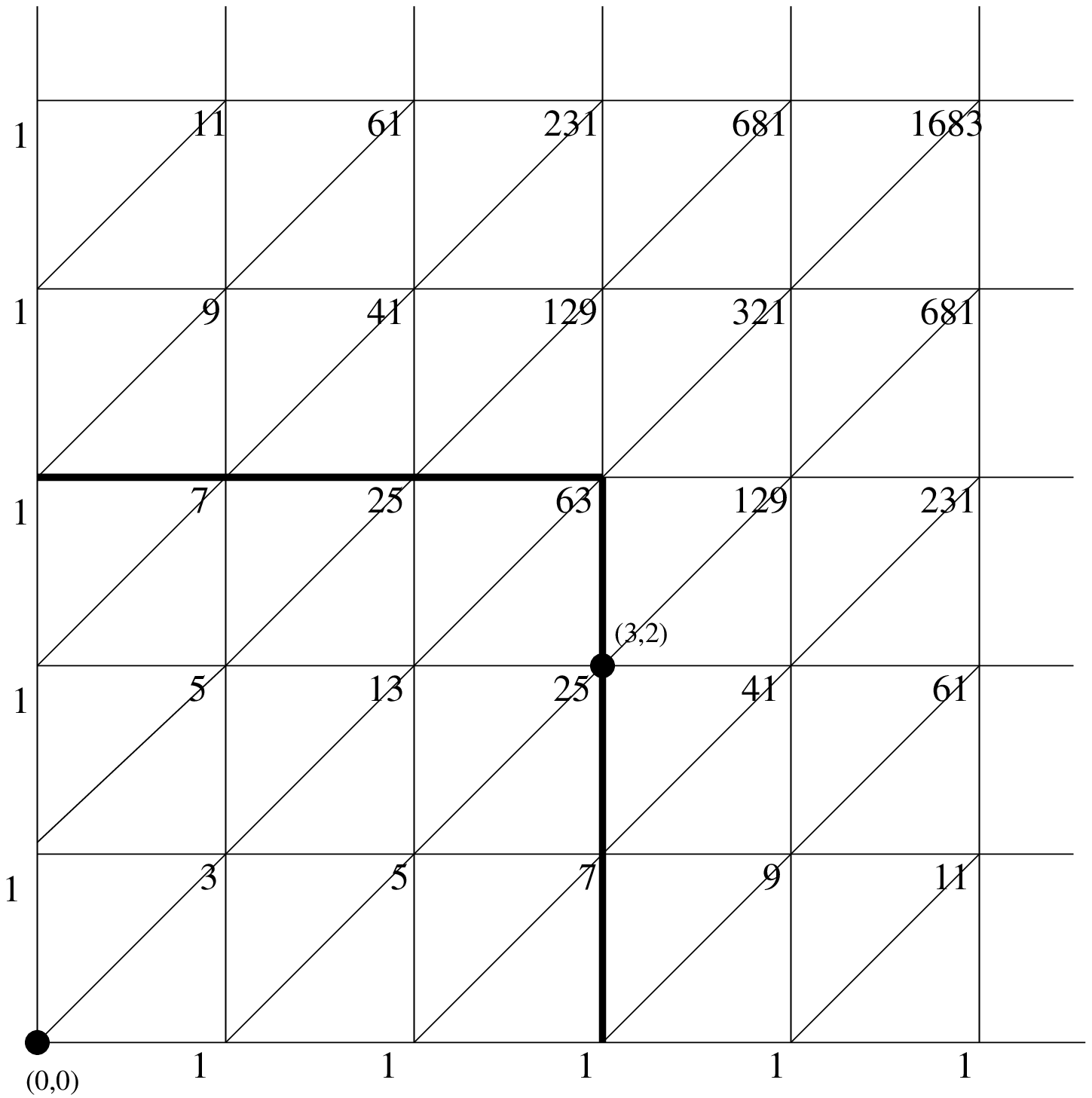}}
  \caption{The Delannoy graph with part of a ``blocking set''.}
  \label{fig:DelBlocking}
  \end{figure}

\section{Dimension groups}\label{sec:DimGps}

We want to compute the dimension group of the Delannoy Bratteli diagram shown in Figure \ref{fig:DelBrat}. The adjacency matrices $A_i$ between levels $i$ and $i+1$, for $i \geq 0$, determine the sequence of ordered abelian group homomorphisms

 \be
 \Bz \overset{A_1}\to \Bz^3 \overset{A_2}\to \Bz^5 \overset{A_3}\to \Bz^7 \overset{A_4}\to \dots .
 \en

Writing as usual transposes of the vectors and matrices involved for ease of typing,

\be
\begin{gathered}
A_0= (1 1 1), \\
A_1=\left(
\begin{matrix}
1 &1 &1 &0 &0 \\
0 &0 &1 &0 &0 \\
0 &0 &1 &1 &1\end{matrix}
\right),\\
A_2=\left(
\begin{matrix}
1 &1 &1 &0 &0 &0 &0 \\
0 &0 &1 &0 &0 &0 &0 \\
0 &0 &1 &1 &1 &0 &0 \\
0 &0 &0 &0 &1 &0 &0 \\
0 &0 &0 &0 &1 &1 &1
\end{matrix}
\right) , \dots , \\
A_{n+1}=\left(
\begin{matrix}
&&&& &0 &0 \\
&&&& &0 &0 \\
&&&&&&& \\
&&A_n &&&&& \\
&&&&&&& \\
0 &0 &\dots &0 &1 &0 &0 \\
0 &0 &\dots &0 &1 &1 &1
\end{matrix}
\right) .
\end{gathered}
\en

The {\em dimension group} $K_0(X,T)$ of the Delannoy adic dynamical system is the direct limit of this system of abelian groups and homomorphisms. It can be represented as the set of equivalence classes of the discrete union of $\{ \Bz^{2k+1} : k =0,1,2,\dots \}$, with $v \in \Bz^{2k+1}$ equivalent to $w \in \Bz^{2{k+j}+1}$ if $A_{k+j} \cdots A_{k+1} v = w$. It has a nonnegative cone consisting of those equivalence classes which contain a nonnegative vector, and a distinguished order unit,  the equivalence class of $ 1 \in \Bz$. Equivalence classes are added by adding representatives that have equal dimensions.

 We can also give a slightly more concrete representation of this dimension group as follows. We correspond to each $v \in \Bz^{2k+1}$ the following pair of polynomials with coefficients from $\Bz$:
 \be
 r(v)=\sum_{i=0}^{k-1} v_{2i} x^i, \qquad s(v)=\sum_{j=0}^k v_{2j+1} x^j .
 \en
 Then $A_{k+1}v$ corresponds to the pair of polynomials
 \be
 (r',s') = (s, xr(v) + (1+x) s(v) = (r,s)
 \left( \begin{matrix}
 0 & x\\
 1 & 1+x
 \end{matrix} \right) .
 \en
 The connecting maps $A_{k+1}$ can be realized by a single matrix
  \be
 B=\left( \begin{matrix}
 0 & x\\
 1 & 1+x
 \end{matrix} \right) ,
 \qquad \text { with inverse  }
 B^{-1} = -\frac{1}{x} \left( \begin{matrix}
 1+x & -x\\
 -1 & 0
 \end{matrix} \right).
 \en
  Thus the discrete union of $\{ \Bz^{2k+1} : k =0,1,2,\dots \}$ maps onto the set of pairs of polynomials $(r,s)$ with integer coefficients and $\deg (r) < \deg (s)$. Two pairs $(r,s)$ and $(u,v)$ are equivalent if there is $n \in \Bz$ such that $(u,v)=B^n (r,s)$. Each nonzero equivalence class contains a pair $(r,s)$ of smallest degree, which satisfies $r(0) \neq s(0)$. It can be found by starting with any member of the class and applying powers of $B^{-1}$ repeatedly. Equivalence classes can be added by adding members of equal degrees.
 Unfortunately the map that takes $v \in \cup \{ \Bz^{2k+1} : k =0,1,2,\dots \}$ to the pair of polynomials $(r(v), s(v))$ is many-to-one: for example, $(-1, 1, 0)$ and $(-1, 1, 0, 0, 0)$ have the same image. (The map could be made one-to-one by sending $v \in \Bz^{2k+1}$ to the triple $(r(v), s(v), k)$.) So the set of equivalence classes of pairs of integer-coefficient polynomials is  a quotient of the dimension group of the Delannoy adic. The distinguished order unit is the class of the image of $1 \in \Bz$, namely the class of the pair of polynomials $r(x)=1, s(x)=x+1$. This class consists of
  the ``Delannoy polynomials'' $P_n(x)$ defined by $P_0(x)=1,P_1(x)=x+1, P_{n+1}(x)=(x+1)P_n(x)+xP_{n-1}(x)$, whose coefficients form the array of Delannoy numbers.

\section{Remarks}
 (1) Recall that $A=\{ h,d,v \}$ denotes the set of edge labels for paths in the Delannoy graph. Denote by $\sigma : A^{\Bz} \to A^{\Bz}$ the usual shift map ($\sigma \omega)_i=(\omega)_{i+1}, i \in \Bz$). For the Delannoy system there is a countable set of paths $X' \subset X$ such that the map $\phi : X \setminus X' \to A^{\Bn}$ defined by $(\phi (x))_i = e_1(T^i x)$ is one-to-one (and Borel measurable) and satisfies $\phi \circ T = \sigma \circ \phi$. Thus for each of its invariant measures $\mu$, the Delannoy system is isomorphic to a genuine topological dynamical system (a compact metric space together with a homeomorphism) with a shift-invariant Borel probability measure, namely the subshift $\Sigma$ which is the closure of the image of $\phi$ together with $\sigma$ and the image of $\mu$ under $\phi$. The proof is the same as in \cite{MelaPetersen2005}.

 (2) The just-mentioned subshift $(\Sigma, \sigma)$ is topologically weakly mixing. The proof is the same as in \cite{MelaPetersen2005}.

 (3) With each of its ergodic invariant measures, the Delannoy adic dynamical system is loosely Bernoulli. The proof is the same as in \cite{delaRueJanvresse2004} (see also \cite{Frick2009}).

 (4) It may be possible to determine for $(\Sigma,\sigma)$ limit laws for return times to (or hitting times of) cylinder sets and the complexity of the associated formal language, as in \cite{MelaPetersen2005}. We do not know, for non-atomic ergodic measures, about weak mixing, multiplicity of the spectrum, or joinings of any $\mu_\alpha$ with any $\mu_\beta$.

(5) The Delannoy adic is another example, after the Pascal, of a system that is {\em isotropic} in the sense that there is the same pattern of edges leaving each vertex. It seems that section (3) of the proof of Theorem \ref{thm:DelInvMeas} (X. M\'{e}la's isotropy argument) will extend to all such systems to identify their ergodic measures. These systems will be the subject of future work.

\end{document}